\documentclass{article}
\usepackage{amsmath,amsthm, amsxtra,amssymb,latexsym, amscd}
\usepackage[mathscr]{eucal}
\usepackage{multicol}
\usepackage{makeidx}
\usepackage{graphics, graphpap}
\newtheorem{theorem}{Theorem}[section]
\newtheorem{corollary}[theorem]{Corollary}
\newtheorem{lemma}[theorem]{Lemma}

\theoremstyle{definition}

\newtheorem{remark}[theorem]{Remark}
\newcommand{\ch}{\ensuremath{\bullet}} %{\centerdot}
\newcommand{\mt}{\ensuremath{\rightarrow}} %{\centerdot}
\newcommand{\fr}{\ensuremath{\frak}} %{\centerdot}
\DeclareMathOperator{\Ass}{Ass}
\DeclareMathOperator{\Assm}{Assm}

\DeclareMathOperator{\ann}{ann}

\DeclareMathOperator{\Supp}{Supp}
\DeclareMathOperator{\depth}{depth}
\DeclareMathOperator{\gdepth}{gdepth}

\DeclareMathOperator{\Ext}{Ext}
\DeclareMathOperator{\Hom}{Hom}
\DeclareMathOperator{\Tor}{Tor}
\DeclareMathOperator{\Att}{Att}
\DeclareMathOperator{\pd}{pd}
\DeclareMathOperator{\injd}{injd}

\begin{document}

\title{ON THE VANISHING AND THE FINITENESS OF SUPPORTS OF  GENERALIZED LOCAL COHOMOLOGY MODULES}         % Enter your title between curly braces
\author{NGUYEN TU CUONG and NGUYEN VAN HOANG}        % Enter your name between curly braces
\date{Institute of Mathematics\\18 Hoang Quoc Viet Road,  10307  Hanoi, Vietnam}
\maketitle

\vskip .5cm {\begin{quote} \normalsize{\bf Abstract}{\footnote{{\it{Key words and phrases: }} Generalized local cohomology, local cohomology, support, associated primes, generalized depth. \hfill\break
{\it{2000 Subject  Classification:}} 13D45, 13C15.\hfill\break {This work is supported in part by
the National Basis Research Programme in Natural Science of Vietnam.}\hfill\break {$^*$ E-mail: ntcuong@math.ac.vn}\hfill\break {$^{**}$  E-mail: nguyenvanhoang1976@yahoo.com}}. Let $(R,\fr m)$ be a Noetherian local ring,  $I$  an ideal of $R$ and $M, N$ two finitely generated $R$-modules.  The first result of this paper is to prove  a vanishing theorem for generalized local cohomology modules which says that  $H^j_I(M,N)=0$ for all $j>\dim(R)$, provided  $M$ is of  finite projective dimension. Next, we  study and give characterizations for the least and the last integer $r$ such that $\Supp(H^r_I(M,N))$ is infinite.}
\end{quote}

\section{Introduction}

For an integer $j\geqslant 0$, the $j^{th}$  generalized local cohomology module $H^j_I(M,N)$ of two $R$-modules $M$ and $N$ with respect to  an ideal $I$  was defined by  J. Herzog in \cite{Her} as follows
$$\displaystyle H^j_I(M,N)=\varinjlim_n\Ext^j_R(M/I^nM,N).$$
It is clear that $H^j_I(R,N)$ is just the  ordinary local cohomology module $H^j_I(N)$ of $N$ with respect to $I$. For a better understanding about local cohomology modules $H^j_I(N)$, ones established some problems on the finiteness of these modules such as determining when the sets $\Supp (H^j_I(N))$ and $\Ass (H^j_I(N))$ to be finite.  There are some affirmative answers for the finiteness of $\Ass (H^j_I(N))$ and of  $\Supp (H^j_I(N))$ ( see  \cite{Hu}, \cite{HS}, \cite{Kh},  \cite{Mar}, \cite{Nh1}...).  But,  there exist in general local cohomology modules of Noetherian local rings with infinitely many associated primes, see \cite{Ka}. 

Similar problems are raised for generalized local cohomology modules. It should be mentioned here that some basic properties of local cohomology modules can not extend to generalized local cohomology modules. For example, if $N$ is $I$-torsion then $H^i_I(N)=0$ for all $i>0$, but $H^i_I(M,N)\cong\Ext^i_R(M,N)$ and the later does not vanish in general for $i>0$; or while the Grothendieck's Vanishing Theorem says that $H^i_I(N)=0$ for all $i>\dim(N)$, the generalized local cohomology modules $H^i_I(M,N)$ may not vanish in general for infinitely many $i\geqslant 0$. However, we can show  in this paper that $\bigcup_{j\leqslant i}\Supp(H^j_I(M,N))=\bigcup_{j\leqslant i}\Supp(\Ext^j_R(M/IM,N))$ for all $i\geqslant 0$ (Lemma \ref{L24}). It follows that although $H^i_I(M,N)$ may not vanish but $\Supp(H^i_I(M,N))\subseteq\bigcup_{j\leqslant\dim(N)}\Supp(H^j_{I_M}(N))$ for all $i\geqslant 0$, where $I_M=\ann(M/IM)$ the annihilator of $R$-module $M/IM$. Moreover, we also prove that if $M$ has finite projective dimension then $H^j_I(M,N)=0$ for all $j>\dim(R)$ (Theorem \ref{L204}). 
The purpose of this paper is to exploit Lemma \ref{L24} and Theorem \ref{L204} in the studying the finiteness of the supports and the finiteness of the set of associated primes of generalized local cohomology modules. 

Our paper is divided into 5 sections. In section 2, 
 we prove two auxiliary lemmas (Lemmas  \ref{L23} and \ref{L24}) and its consequence (Corollary \ref{C25}) on the support of generalized local cohomology modules.  In section 3, by using spectral sequences, we prove that  $H^j_I(M,N)=0$ for all $j>\dim(R)$, provided  $M$ is of  finite projective dimension (Theorem \ref{L204}). This  generalizes a vanishing result of generalized local cohomology modules with respect to the maximal ideal of  J. Herzog and N. Zamani   \cite[Theorem 3.2]{HZ}. In section 4, we use Lemma \ref{L24} and the notion of  generalized regular sequences introduced by Nhan \cite {Nh1} to characterize the least integer $r$ such that $\Supp(H^r_I(M,N))$ is an infinite set (Theorem \ref{T24}); from this we can describe concretely the finiteness of $\Ass(H^r_I(M,N))$ (Theorem \ref{T28}). In the last section, we study the last integer $s$ such that $\Supp(H^s_I(M,N))$ is an infinite set (Theorem \ref{T34}$(a)$); and we also  give lower and upper  bounds for $s$ (Theorem \ref{T34}$(b)$).

\section{Preliminaries}
Throughout this paper $M, N$ are finitely generated modules over  a Noetherian local ring $(R, \fr m)$. Let $\pd_R(M)$ denote the projective dimension of $M$. For any ideal $I$ of  $R$  we denote by $I_M=\ann_R(M/IM)$ the annihilator of the module $M/IM$ and by $\Gamma_I$ the $I$-torsion functor.  
First, we recall some known facts on generalized local cohomology modules.

\begin{lemma}\label{L21} {\rm(cf. \cite[Lemmas 2.1, 2.3]{CH})} The following statements are true.

\item[\quad$(i)$] Let $E^{\ch}$ be an injective resolution of $N$. Then, for any $j\geqslant 0$, we have \begin{alignat}{2}H^j_I(M,N)&\cong H^j(\Gamma_I(\Hom(M,E^{\ch})))\notag\\
&\cong H^j(\Hom(M,\Gamma_I(E^{\ch})))\cong H^j(\Hom(M,\Gamma_{I_M}(E^{\ch}))).\notag
\end{alignat}
\item[\quad$(ii)$] If $\Gamma_{I_M}(N)=N$ or $I\subseteq\ann(M)$, then $H^j_I(M,N)\cong\Ext^j_R(M,N)$ for all $j\geqslant 0$.
\end{lemma}

\begin{lemma}\label{L22} {\rm(cf. \cite[Theorem 2.4]{CH})} Let $l=\depth(I_M,N)$. Then
$$\Ass H^l_I(M,N)=\Ass\Ext^l_R(M/IM,N).$$
\end{lemma}

\begin{lemma}\label{ya} {\rm(cf. \cite[Theorem 3.7]{Ya1})} If $\pd_R(M)<+\infty$, then $H^j_I(M,N)=0$ for all $j>\pd_R(M)+\dim(M\otimes_RN)$.
\end{lemma}

\begin{lemma} \label{L202} {\rm (cf. \cite[Lemma 3.1]{HZ})} Let  $d=\dim(R)$. If $\pd_R(M)<+\infty$, then $\dim(\Ext^j_R(M,R))\leqslant d-j$ for all $0\leqslant j\leqslant\pd_R(M)$.
\end{lemma}

\begin{lemma}\label{rm3} Assume that the local ring homomorphism $f: R\mt S$ is flat. Then $H^j_I(M,N)\otimes_RS\cong H^j_{IS}(M\otimes_RS,N\otimes_RS)$ for all $j\geqslant 0$.
\end{lemma}

\begin{lemma}\label{rm1} Let $n=\dim(N)$. Then $\Supp(H^{n-1}_I(N))$ is a finite set.
\end{lemma}
\begin{proof} Let $\fr a=\ann_R(N)$ and $\overline R=R/\fr a$, then $\dim(\overline R)=n$ and $N$ is an $\overline R$-module. Hence, by the independence theorem in \cite{Bs}, we have $H^{n-1}_I(N)\cong H^{n-1}_{I\overline R}(N)$ as $R$-modules. By \cite[Corollary 2.5]{Mar}, we obtain that $\Supp_{\overline R}(H^{n-1}_{I\overline R}(N))$ is finite. On the other hand, we have $\Supp(H^{n-1}_I(N))\subseteq\Supp(R/\fr a)$ and
$$\Supp_{\overline R}(H^{n-1}_{I\overline R}(N))=\{\fr p/\fr a\mid\fr p\in\Supp(H^{n-1}_I(N))\}.$$Therefore $\Supp(H^{n-1}_I(N))$ is a finite finite, as required.
\end{proof}

The next two lemmata are  impotant for our further investication in this paper.

\begin{lemma}\label{L23} Let $\Bbb N$ be the set of all positive integers and  $i\in\Bbb N\cup\{+\infty\}$. Set  $J_i=\bigcap_{j<i}\ann(\Ext^j_R(M/IM,N))$. Then $H^j_I(M,N)\cong H^j_{J_i}(M,N)$ for all $j<i.$
\end{lemma}
\begin{proof}We note first that $I_M\subseteq J_i$. Let  $E^{\bullet}: 0\mt E^0\mt\cdots\mt E^j\mt\cdots$ be a minimal injective resolution of $N.$ For any $j\geqslant 0$, we have by \cite[10.1.10]{Bs} that
$$\Gamma_{I_M}(E^j)=\bigoplus_{I_M\subseteq\fr p\in\Ass(E^j)}E(R/\fr p)^{\mu^j(\fr p,N)}$$
and
$$\Gamma_{J_i}(E^j)=\bigoplus_{J_i\subseteq\fr p\in\Ass(E^j)}E(R/\fr p)^{\mu^j(\fr p,N)},$$
where $\mu^j(\fr p,N)=\dim_{k(\fr p)}(\Ext^j_{R_{\fr p}}(k(\fr p),N_{\fr p}))$ is $j$th Bass number of $N$ with respect to $\fr p$.  Note that for any $\fr p\in\Ass(E^j)$ the sequence
$0\mt E^0_{\fr p}\mt E^1_{\fr p}\mt\cdots\mt E^j_{\fr p}\mt\cdots$
 is a minimal injective resolution of $N_{\fr p}$ (cf. \cite[11.1.6]{Bs}). So, as $E^j_{\fr p}\not=0$, $N_{\fr p}\not=0$. Consider now two cases: $i\in\Bbb N$ and $i=+\infty$.

Let $i\in\Bbb N$. For any $j<i$, and any $\fr p\in\Ass(E^j)$ satisfying $\fr p\supseteq I_M$ and $\fr p\nsupseteq J_i$, we have $\Ext^l_R(M/IM,N)_{\fr p}=0$ for all $l< i$. It implies $\depth((I_M)_{\fr p},N_{\fr p})\geqslant i,$ so $\depth(N_{\fr p})\geqslant i.$ Thus $\mu^j(\fr p,N)=0$, so that $\Gamma_{I_M}(E^j)=\Gamma_{J_i}(E^j)$. Hence, by Lemma \ref{L21}, we get $H^j_I(M,N)\cong H^j_{J_i}(M,N)$ for all $j< i$. 

Finally, if $i=+\infty$, then $J_i=\bigcap_{j\geqslant 0}\ann(\Ext^j_R(M/IM,N))$. For any $j\geqslant 0$, and any $\fr p\in\Ass(E^j)$ such that $\fr p\supseteq I_M$, we obtain $(I_M)_{\fr p}N_{\fr p}\not=N_{\fr p}.$ Set $\nu=\depth((I_M)_{\fr p},N_{\fr p})$, then  $\nu<+\infty$ and $\fr p\in\Supp(\Ext^{\nu}_R(M/IM,N))$. It follows that  $\fr p\supseteq\ann(\Ext^{\nu}_R(M/IM,N))\supseteq J_i,$ and hence $\Gamma_{I_M}(E^j)=\Gamma_{J_i}(E^j)$. Therefore, by Lemma \ref{L21}, we obtain $H^j_I(M,N)\cong H^j_{J_i}(M,N)$ for all $j\geqslant 0$.
\end{proof}

\begin{lemma}\label{L24} Let $i\in\Bbb N\cup\{+\infty\}$. Then we have $$\bigcup_{j<i}\Supp (H^j_I(M,N))=\bigcup_{j<i}\Supp(\Ext^j_R(M/IM,N)).$$
\end{lemma}
\begin{proof} Let $i\in\Bbb N\cup\{+\infty\}$, and $J_i=\bigcap_{j<i}\ann(\Ext^j_R(M/IM,N))$. Then, by Lemma \ref{L23}, we obtain $H^j_I(M,N)\cong H^j_{J_i}(M,N)$ for all $j< i$. So, we have 
$$\bigcup_{j<i}\Supp (H^j_I(M,N))\subseteq\Supp(R/J_i)=\bigcup_{j<i}\Supp(\Ext^j_R(M/IM,N)).$$

Conversely, let $\fr p\in\bigcup_{j<i}\Supp(\Ext^j_R(M/IM,N))$. Set $\nu=\depth((I_M)_{\fr p},N_{\fr p})$, then $\nu<i$. For each $n>0$,  $I_M\subseteq\sqrt{\ann(IM/I^nM)}$, so that $\Ext^j_R(IM/I^nM,N)_{\fr p}=0$ for all $j<\nu$. Thus the following sequence $$0\mt\Ext^{\nu}_R(M/IM,N)_{\fr p}\mt\Ext^{\nu}_R(M/I^nM,N)_{\fr p}$$ is exact for all $n>0$. This induces an exact sequence $$0\mt\Ext^{\nu}_R(M/IM,N)_{\fr p}\mt H^{\nu}_I(M,N)_{\fr p}.$$
So, since $\Ext^{\nu}_R(M/IM,N)_{\fr p}\not=0$ and $\nu<i$, we obtain that $\fr p\in\bigcup_{j<i}\Supp(H^j_I(M,N))$ as required.
\end{proof}

Until now one does not know about the last integer $i$ such that $H^i_I(M,N)\not=0$, even for the case $I=\fr m$. For example, assume that $R$ is not regular local ring and $\pd_R(M)=+\infty$, then $H^i_{\fr m}(M,R/\fr m)=\Ext^i_R(M,R/\fr m)\not=0$ for all $i\geqslant 0$.  However, the following result shows that there is a union of only finitely many supports of generalized local cohomology modules so that the other supports can be viewed as its subsets.

\begin{corollary}\label{C25} Let $n=\dim(N)$. Then we have $$\bigcup_{j\geqslant 0}\Supp(H^j_I(M,N))=\bigcup_{j\leqslant n}\Supp (H^j_{I_M}(N))=\bigcup_{j\leqslant n}\Supp (H^j_I(M,N)).$$
\end{corollary}
\begin{proof} For any $i\in\Bbb N\cup\{+\infty\}$, by Lemma \ref{L24}, we obtain that $$\bigcup_{j<i}\Supp (H^j_{I_M}(N)) =\bigcup_{j<i}\Supp(\Ext^j_R(R/I_M,N)).$$ Moreover, by using basic properties of regular sequences, we get $$\bigcup_{j<i}\Supp (\Ext^j_R(R/I_M,N))=\bigcup_{j<i}\Supp(\Ext^j_R(M/IM,N)).$$ Thus, by Lemma \ref{L24}, we have $$\bigcup_{j<i}\Supp (H^j_{I_M}(N))=\bigcup_{j<i}\Supp(H^j_I(M,N))$$
and the corollary follows by Grothendieck's  Vanishing Theorem. 
\end{proof}
 
\section{A vanishing theorem}

J.  Herzog and N. Zamani  showed in \cite[Theorem 3.2]{HZ}  that if $\pd_R(M)<+\infty$, then $H^t_{\fr m}(M,N)=0$ for all $t>\dim(R)$. In this  section we extence Herzog-Zamani's result for an arbitrary ideal $I$ as follows.
\begin{theorem}\label{L204} Let $d=\dim(R)$. Assume that $\pd_R(M)<+\infty$, then  $H^t_I(M,N)=0$ for all $t>d$.
\end{theorem}
\begin{proof} We first claim that $H^t_I(M,R)=0$ for all $t>d$. 
 Let $x_1,\ldots,x_m$ be a set of generators of $I$ and  $K^{n}_{\bullet}$ the Koszul complex of $R$ with respect to $x_1^n,\ldots,x_m^n$. We denote  by $C^n_{\bullet}$  the total complex associated to the double complex $K^n_{\bullet}\otimes_RF_{\bullet}$, where  $F_{\bullet}$ is a projective resolution of $M$. Consider the convergent spectral sequence
$$H^i(\Hom(K^{n}_{\bullet},\Ext^j_R(M,N)))\underset{i}{\Longrightarrow} H^{i+j}(\Hom(C^n_{\bullet}, N)).$$
Since $H^i_I(M,N)\cong\varinjlim_{n}H^i(\Hom(C^n_{\bullet},N))$ for all $i\geqslant 0$ by \cite[Theorem 4.2]{Za}, we obtain by passage to direct limits the following convergent spectral sequence
$$E^{i,j}_2:=H^i_{I}(\Ext^j_R(M,R))\underset{i}{\Longrightarrow}H^{i+j}= H^{i+j}_{I}(M,R).$$
Thus, for each $t\geqslant 0$, there is a finite filtration of the module $H^t=H^t_I(M,R)$
\begin{alignat}{2}
0=\phi^{t+1}H^t\subseteq\phi^{t}H^t\subseteq&\ldots\subseteq\phi^1H^t\subseteq\phi^0H^t=H^t\notag
\end{alignat}
such that $E^{i,t-i}_{\infty}\cong\phi^iH^t/\phi^{i+1}H^t$ for all $0\leqslant i\leqslant t$. It is clear that $E^{i,j}_2=0$ for all $j>\pd_R(M)$. If $i+j>d$ then $i>d-j$. Hence, by Lemma \ref{L202}, $i>\dim(\Ext^j_R(M,R))$ for all $0\leqslant j\leqslant\pd_R(M)$. So $E^{i,j}_2=0$ for all $0\leqslant j\leqslant\pd_R(M)$ and $i+j>d$. Thus, for each $t>d$, we get  $E^{i,t-i}_2=0$ for all $0\leqslant i\leqslant t$. Moreover, since $E^{i,t-i}_{\infty}$ is subquotient of $E^{i,t-i}_2$ for all $0\leqslant i\leqslant t$, it implies that  $E^{i,t-i}_{\infty}=0$ for all $0\leqslant i\leqslant t$ and all $t>d$. Therefore by  the  exact sequences
$$0\mt\phi^{i+1}H^t\mt\phi^{i}H^t\mt E^{i,t-i}_{\infty}\mt 0$$
for all $0\leqslant i\leqslant t$ we have $H^t_I(M,R)=\phi^0H^t=H^t=0$, and the claim is proved.
\medskip

Next,  since $d+1>\dim_R(M\otimes_RN),$ we get by Lemma \ref{ya} that $H^t_I(M,N)=0$ for all $t>\pd_R(M)+d+1$. Thus, it is enough to prove by descending induction on $t$ that $H^t_I(M,N)=0$, for all $d<t\leqslant\pd_R(M)+d+1$.  It is clear that  the assertion is true for $t=\pd_R(M) +d+1$. Assume that $d<t<\pd_R(M)+d+1$ and the assertion is true for $t+1$. For each finitely generated $R$-module $N$, there exists a non negative integer $n$ such that the following sequence 
$$0\mt L\mt R^n\mt N\mt 0$$
is exact for some finitely generated $R$-module $L$. It induces an exact sequence of generalized local cohomology modules
$$H^t_I(M,R^n)\mt H^t_I(M,N)\mt H^{t+1}_I(M,L).$$
By inductive hypothesis, we get $H^{t+1}_I(M,L)=0$. Since $t>d$, we get by the claim  that $H^t_I(M,R^n)=H^t_I(M,R)^n=0$.  Therefore $H^t_I(M,N)=0$ as required.
\end{proof}

It should be mentioned that the functor $H^t_I(M,-)$ commutes with the direct limits in the category of all $R$-modules. Therefore, as an immediate consequence of Theorem \ref{L204} we get the following  result. 

\begin{corollary} Let  $d=\dim(R)$. Assume that $\pd_R(M)<+\infty$, then we have $H^t_I(M,K)=0$ for all $t>d$ and all (not necessary to be finitely generated) $R$-module $K$.
\end{corollary}

\begin{remark}
It is well-known by the Vanishing and non Vanishing Theorem of Grothendieck for  the theory of ordinary local cohomology that $H_{I}^t(N) = 0$ for all $t> n= \dim N$ and $H_{\fr m}^n(N)\not= 0$. Therefore, 
 in view of Theorem \ref{L204}, it is natural to ask whether  $H^i_I(M,N)=0$ and $H_{\fr m}^n(M,N)\not= 0$ for all $i>n= \dim N$ and all finitely generated $R$-modules of finite projective dimension. Unfortunately, the answer is negative as the  following example illustrates.
\medskip

\noindent Let $k$ be a field and $R=k[[x,y,u,v]]$. Let $\fr m=(x,y,u,v)R$, $M=R/(y)$ and $N=R/(x)\cap (y)$. It is clear that $\dim(R)=4$, $\pd_R(M)<+\infty$ and $\dim(N)=3$.  Following \cite{HZ}, from the exact sequence $0\mt R\xrightarrow{y}R\mt M\mt 0$, we get an exact sequence $$H^3_{\fr m}(N)\xrightarrow{y}H^3_{\fr m}(N)\mt H^4_{\fr m}(M,N)\mt 0.$$
Thus, we get an isomorphism $H^4_{\fr m}(M,N)\cong H^3_{\fr m}(N)/yH^3_{\fr m}(N)$. By \cite[Theorem 7.3.2]{Bs}, we have $\Att(H^3_{\fr m}(N))=\{(x)R, (y)R\}$. Note that $H^3_{\fr m}(N)$ is Artinian. Hence, since $y\in (y)R$, we get by \cite[Proposition 7.2.11]{Bs} that $yH^3_{\fr m}(N)\not=H^3_{\fr m}(N)$. It follows that $H^4_{\fr m}(M,N)\not=0$. 
\medskip
\end{remark}

\section{The least integer $r$ such that $\Supp(H^r_I(M,N))$ is infinite}

First, we recall the notion of generalized regular sequences introduced in \cite{Nh1}: A sequence $x_1,\ldots,x_r\in I$ is called a {\it generalized regular sequence} of $N$ in $I$ if for any $j=1,\ldots,r$, $x_j\notin\fr p$ for all $\fr p\in\Ass(N/(x_1,\ldots,x_{j-1})N)$ satisfying $\dim (R/\fr p)\geqslant 2$. If $\dim(N/IN)\geqslant 2$ then any generalized regular sequence of $N$ in $I$ is of length at most $\dim(N)-\dim(N/IN),$ all maximal generalized regular sequences of $N$ in $I$ have the same length, and this common length is called generalized depth of $N$ in $I$ and denoted by $\gdepth(I,N)$. Note that if $\dim (N/IN)\leqslant 1$ then there exists a generalized regular sequence of length $r$  of $N$ in $I$ for any given integer $r>0.$ So, in this case we stipulate $\gdepth(I,N)=+\infty.$ Below we show that the generalized depth can be computed by  generalized local cohomology modules.

\begin{theorem}\label{T24} Set $r=\gdepth(I_M,N)$ and
$J_r=\bigcap_{j<r}\ann(\Ext^j_R(M/IM,N)).$
Then $\dim(R/J_r)\leqslant 1$ and 
\begin{alignat}{2} r&=\inf\{i\mid\Supp(H^i_I(M,N))\text{ is not finite}\}\notag\\
&=\inf\{i\mid H^i_I(M,N)\ncong H^i_{J_r}(M,N)\},\notag
\end{alignat}
where we use the convenience that $\inf(\emptyset)=+\infty.$
\end{theorem}

\begin{proof} If $\dim(N/I_MN)\leqslant 1$ then $r=+\infty$ and $\dim(R/J_r)\leqslant 1$. Since 
 $\Supp(N/I_MN)$ is finite, $\Supp(H^j_I(M,N))$ is finite for all $j\geqslant 0$. On the other hand, by Lemma \ref{L23}, $H^j_I(M,N)\cong H^j_{J_r}(M,N)$ for all $j\geqslant 0$. Therefore, the result is true in this case. 

If $\dim(N/I_MN)\geqslant 2$ then $r<+\infty$. Let $x_1,\ldots,x_r$ be a maximal generalized regular sequence of $N$ in $I_M$. For any $\fr p\in\Supp (N/I_MN)$ such that $\dim (R/\fr p)\geqslant 2$, we obtain that $x_1/1,\ldots,x_r/1$ is an $N_{\fr p}$-regular sequence in $(I_M)_{\fr p}$. It follows that $\Ext^j_R(M/IM,N)_{\fr p}=0$ for all $j<r$, hence $\dim(\Ext^j_R(M/IM,N))\leqslant 1$ for all $j<r$. Thus $\dim(R/J_r)\leqslant 1$, so that $\bigcup_{j<r}\Supp(\Ext^j_R(M/IM,N))=\Supp(R/J_r)$ is a finite set. From this, we obtain by Lemma \ref{L24} that $\Supp(H^j_I(M,N))$ is a finite set for all $j<r$. Note that, by  \cite[Proposition 4.4]{Nh1}, we have $$r=\min\{\depth((I_M)_{\fr p},N_{\fr p})\mid \fr p\in\Supp(N/I_MN),  \dim(R/\fr p)\geqslant 2\}.$$
So $r=\depth((I_M)_{\fr p},N_{\fr p})$ for some $\fr p\in\Supp (N/I_MN)$ with $\dim (R/\fr p)\geqslant 2$. Hence $\Ext^r_R(M/IM,N)_{\fr p}\not=0$. Thus, by Lemma \ref{L24}, we have $\fr p\in\bigcup_{j=0}^r\Supp(H^j_I(M,N))$. But, since $\Supp(H^j_I(M,N))$ is finite for all $j<r$, $\fr p\notin\bigcup_{j<r}\Supp(H^j_I(M,N))$. Thus $\fr p\in\Supp(H^r_I(M,N))$, so that  $\Supp(H^r_I(M,N))$ is an infinite set by \cite[Theorem 144]{Kap}. Therefore $$r=\inf\{i\mid\Supp(H^i_I(M,N))\text{  is not finite }\}.$$

Finally, keep in mind that  $\dim(R/J_r)\leqslant 1$, hence $\Supp(H^r_{J_r}(M,N))$ is finite; while $\Supp(H^r_I(M,N))$ is infinite by the above equality. Thus $H^r_I(M,N)\ncong H^r_{J_r}(M,N)$. Moreover, by Lemma \ref{L23}, we get $H^j_I(M,N)\cong H^j_{J_r}(M,N)$ for all $j<r$. Therefore $$r=\inf\{i\mid H^i_I(M,N)\ncong H^i_{J_r}(M,N)\}$$ as required.
\end{proof}

\begin{corollary}\label{C215} Let $i$ be a non negative integer. If $\Supp(H^j_I(N))$ is finite for all $j\leqslant i$ then so is $\Supp(H^j_I(M,N))$ for all finitely generated $R$-module $M$. 
\end{corollary}
\begin{proof} As $\gdepth(I,N)\leqslant\gdepth(I_M,N)$, the result follows by Theorem \ref{T24}.
\end{proof}

Note that for an arbitrary $R$-module $K$, the condition for $\Supp (K)$ to be a finite set  is in general not equivalent to the condition that $\dim (R/\fr p)\leqslant 1$ for all $\fr p\in\Supp (K).$  However, we have  immediate  consequences of Theorem \ref{T24} as follows. 

\begin{corollary} Let $i$ be a non negative integer. Then $\Supp(H^j_I(M,N))$ is finite for all $j\leqslant i$ if and only if $\dim(R/\fr p)\leqslant 1$ for all $\fr p\in\Supp(H^j_I(M,N))$ and all $j\leqslant i$. 
\end{corollary}

\begin{corollary}\label{c213} Set $r=\gdepth(I,N)$ and $J_r=\bigcap_{j<r}\ann(\Ext^j_R(R/I,N))$. Then $\dim(R/J_r)\leqslant 1$ and 
$$r=\inf\{i\mid\Supp(H^i_I(N))\text{ is not finite }\}=\inf\{i\mid H^i_I(N)\ncong H^i_{J_r}(N)\}.$$
\end{corollary}

It should be mentioned that the first equality of Corollary \ref{c213} was proved by L.T. Nhan in \cite[Proposition 5.2]{Nh1}. 

\begin{theorem}\label{T28} Let $i$ be a non negative integer and $P_i=\bigcup_{j<i}\Supp(H^j_I(M,N))$, then $$\Ass(H^i_I(M,N))\bigcup P_i=\Ass(\Ext^i_R(M/IM,N))\bigcup P_i.$$
In particular, $\Ass(H^r_I(M,N))$ is a finite set, where $r=\gdepth(I_M,N)$.
\end{theorem}
\begin{proof}   Let $\fr p\in\Ass(H^i_I(M,N))$. Assume that $\fr p\notin P_i$. Then, by Lemma \ref{L24}, we have $\Ext^i_R(M/IM,N)_{\fr p}\not=0 $ and $\Ext^j_R(M/IM,N)_{\fr p}=0$ for all $j<i.$
 It follows that $i=\depth((I_M)_{\fr p},N_{\fr p}).$  So,  by Lemma \ref{L22}, we get $$\Ass(H^i_I(M,N)_{\fr p})=\Ass(\Ext^i_R(M/IM,N)_{\fr p}).$$
Thus $\fr p\in\Ass(\Ext^i_R(M/IM,N))$, since $\fr pR_{\fr p}\in\Ass(H^i_I(M,N)_{\fr p})$.  Conversely, let $\fr p\in\Ass(\Ext^i_R(M/IM,N))$ and $\fr p\notin P_i$. By similar arguments as above, we can show that $i=\depth((I_M)_{\fr p},N_{\fr p})$. Hence $\fr p\in\Ass(H^i_I(M,N))$ by Lemma \ref{L22}. Therefore  $$\Ass(H^i_I(M,N))\bigcup P_i=\Ass(\Ext^i_R(M/IM,N))\bigcup P_i.$$

Finally, let $r=\gdepth(I_M,N)$. Then we get by Theorem \ref{T24} that $P_r$ is a finite set. Hence $\Ass(H^r_I(M,N))$ is a finite set as required.
\end{proof}

It has shown by \cite[Theorem B]{Kh} or \cite[Theorem 5.6]{Nh1} that if $i$ is an integer such that $\Supp(H^j_I(N))$ is a finite set for all $j<i$ then $\Ass(H^i_I(N))$ is a finite set. The next corollary gives us a description concretely of  this set $\Ass(H^i_I(N))$.

\begin{corollary}\label{c2} Let $i$ be a non negative integer and $P_i=\bigcup_{j<i}\Supp(H^j_I(N))$, then 
$$\Ass(H^i_I(N))\bigcup P_i=\Ass(\Ext^i_R(R/I,N))\bigcup P_i.$$
In particular, $\Ass(H^r_I(N))$ is finite for $r=\gdepth (I,N).$ 
\end{corollary}

\section{The last integer $s$ such that $\Supp(H^s_I(M,N))$ is infinite }

The following theorem is the main result in this section. 

\begin{theorem}\label{T34} Let $s$ be an integer. Assume that $\pd_R(M)<+\infty$. Then 
\item$(a)$ The following statements are equivalent:
\begin{itemize}
\item[$(i)$] $\Supp(H^j_I(M,R/\fr p))$ is finite for all $j>s$ and all $\fr p\in\Assm(N)$, where $\Assm(N)$ denote the set of minimal elements of $\Ass(N)$;
\item[$(ii)$] $\Supp(H^j_I(M,N))$ is finite for all $j>s$;
\item[$(iii)$] $\Supp(H^{s+1}_I(M,R/\fr p))$ is finite for all $\fr p\in\Supp(N)$.
\end{itemize}
\item$(b)$ Assume that $\dim(N/I_MN)\geqslant 2$. Set $d=\dim(R)$, $r=\gdepth(I_M,N)$ and
$$\gamma=\sup\{\pd_{R_{\fr p}}(M_{\fr p})\mid\fr p\in\Supp(N/I_MN), \dim(R/\fr p)\geqslant 2)\}.$$

Let $s$ be the least integer satisfying one of three equivalent conditions in $(a)$, then
$$\max\{r, \gamma\}\leqslant s<d-1.$$
\end{theorem}

\noindent To prove Theorem \ref{T34}$(a)$, we need the following lemma.

\begin{lemma}\label{L221}Assume that $\pd_R(M)<+\infty$. Let $s$ be a non negative integer and $L$ a finitely generated $R$-module such that $\Supp(L)\subseteq\Supp(N)$. Then,  if  $\Supp(H^j_I(M,N))$  is a finite set for all $j>s$, so is $\Supp(H^j_I(M,L))$. 
\end{lemma}

\begin{proof} Let $d=\dim(R)$. By Theorem \ref{L204}, $H^j_I(M,L)=0$ for all $j>d$. Thus, we can assume that $s\leqslant d$. We proceed by descending induction on $j$. It is clear that  the assertion is true for $j\geqslant d+1$. Let $j<d+1$. Since $\Supp(L)\subseteq\Supp(N)$, we get by \cite[Theorem 4.1]{Va} that there exists a finite filtration
$$0=L_0\subset L_1\subset L_2\subset\cdots\subset L_t=L$$
such that for any $i=1,\ldots ,t,$  $L_i/L_{i-1}$ is a homomorphic image of $N^{n_i}$ for some integer $n_i>0$. Using short exact sequences $0\mt L_{i-1}\mt L_i\mt L_i/L_{i-1}\mt 0$ for $i=1,\ldots ,t,$  we can reduce the  situation to the case $t=1$. Therefore, there is an exact sequence $0\mt U\mt N^{n}\mt L\mt 0$ for some $n>0$ and some finitely generated $R$-module $U.$ So, we get a long exact sequence 
$$\cdots\mt H^j_I(M,N^n)\mt H^j_I(M,L)\mt H^{j+1}_I(M,U)\mt\cdots.$$
As $\Supp(U)\subseteq\Supp(N)$, we get  by induction that $\Supp(H^{j+1}_I(M,U))$ is finite. On the other hand, by the hypothesis, $\Supp(H^j_I(M,N^n))=\Supp(H^j_I(M,N))$ is finite. Therefore, by the above exact sequence, $\Supp(H^j_I(M,L))$ is finite as required.
\end{proof}

As an immediate consequence of Lemma \ref{L221}, we get the following result.

\begin{corollary}\label{c35} Assume that $\pd_R(M)<+\infty$. Let $s$ be a non negative  integer and $L$ a finitely generated $R$-module such that $\Supp(L)=\Supp(N)$. Then $\Supp(H^j_I(M,L))$ is finite for all $j>s$ if and only if so is $\Supp(H^j_I(M,N))$.
\end{corollary}

Now, it is ready to prove Theorem \ref{T34}$(a)$.
\medskip

\noindent{\bf Proof of Theorem \ref {T34}}$(a)$. 

\noindent
$(i)\Rightarrow(ii).$ Assume that $\Assm(N)=\{\fr p_1,\ldots,\fr p_t\}$, so that $\Supp(H^j_I(M,R/\fr p_i))$ is finite for all $i=1,\ldots, t$. Set $L_0=0$ and $L_i=\oplus_{j=1}^i(R/\fr p_j)$ for each $i\in\{1,\ldots,t\}$. Note that since $\Ass(L_t)=\{\fr p_1,\ldots,\fr p_t\}=\Assm(N)$,   $\Supp(L_t)=\Supp(N)$.  It follows from Corollary \ref{c35} that $\Supp(H^j_I(M,N))$ is finite for all $j>s$ if we can show that $\Supp(H^j_I(M,L_t))$ is finite for all $j>s$. We do it now by induction on $t$. It is nothing to prove for $t=1$. Let $t>1$. From the exact sequence $0\mt L_{t-1}\mt L_t\mt R/\fr p_t\mt 0$
we get a long exact sequence
$$H^j_I(M,L_{t-1})\mt H^j_I(M,L_t)\mt H^j_I(M,R/\fr p_t).$$
Therefore $\Supp(H^j_I(M,L_t))$ is finite by (i) and the inductive hypothesis.
\medskip

\noindent
$(ii)\Rightarrow (iii)$ follows by Lemma \ref{L221}.
\medskip

\noindent
$(iii)\Rightarrow(i)$. By inductive method we need only to show that  $\Supp(H^{s+2}_I(M,R/\fr p))$ is finite for all $\fr p\in\Supp(N)$. Let $\fr p\in\Supp(N)$. If $\dim(R/\fr p)\leqslant 1$ then the finiteness of $\Supp(H^{s+2}_I(M,R/\fr p))$ is clear. Assume that $\dim(R/\fr p)\geqslant 2$. We consider two cases $I_M\nsubseteq\fr p$ and $I_M\subseteq\fr p$, where $I_M=\ann_R(M/IM)$ the annihilator of the module $M/IM$.

\noindent
Case 1: $I_M\nsubseteq\fr p$.  Then there exists  an $x\in I_M\setminus\fr p$. Set $G=R/(\fr p+xR)$, then $\Supp(G)\subseteq\Supp(N)$. We have a finite filtration
$$0=G_0\subset G_1\subset G_2\subset\cdots\subset G_t=G$$
such that, for each $i=1,\ldots,t$, $G_i/G_{i-1}\cong R/\fr p_i$ for some $\fr p_i\in\Supp(N)$. 
Now, with the same method that used in the proof of  $(i)\Rightarrow(ii)$ we can show that $$\Supp(H^{s+1}_I(M,G))=\Supp(H^{s+1}_I(M,G_t))$$ is a finite set. 
On the other hand, we derive from the exact sequence $0\mt R/\fr p\xrightarrow{x} R/\fr p\mt G\mt 0$ an exact sequence 
$$H^{s+1}_I(M,G)\mt (0:x)_{H^{s+2}_I(M,R/\fr p)}\mt 0.$$
Thus,  $\Supp((0:x)_{H^{s+2}_I(M,R/\fr p)})$ and so $\Supp((0:I_M)_{H^{s+2}_I(M,R/\fr p)})$ is finite. Moreover, since $H^{s+2}_I(M,R/\fr p)$ is $I_M$-torsion by Lemma \ref{L21},   $$\Supp(H^{s+2}_I(M,R/\fr p))=\Supp((0:I_M)_{H^{s+2}_I(M,R/\fr p)})$$ is finite. 

\noindent
Case 2: $I_M\subseteq\fr p$. Then, by Lemma \ref{L21}, we have $$H^j_I(M,R/\fr p)\cong\Ext^j_R(M,R/\fr p)$$ for all $j\geqslant 0$. For  any $\fr q\in\Supp(R/\fr p)$ with $\dim(R/\fr q)\geqslant 2$, we get by our assumption that $$\Supp(\Ext^{s+1}_R(M,R/\fr q))=\Supp(H^{s+1}_I(M,R/\fr q))$$ is finite. Therefore 
$$\Ext^{s+1}_{R_{\fr q}}(M_{\fr q},k(\fr q))=\Ext^{s+1}_R(M,R/\fr q)_{\fr q}=0.$$
Then, by \cite[Proposition 1.3.1]{Bh}, we have $\Tor_{s+1}^{R_{\fr q}}(M_{\fr q},k(\fr q))=0$ . It follows by \cite[Corollary 19.5]{Ei} that $\pd_{R_{\fr q}}(M_{\fr q})<s+1.$ Hence 
$$\Ext^{s+2}_R(M,R/\fr p)_{\fr q}=\Ext^{s+2}_{R_{\fr q}}(M_{\fr q},(R/\fr p)_{\fr q})=0$$
 for all $\fr q\in\Supp(R/\fr p)$ satisfying $\dim(R/\fr q)\geqslant 2$. Thus, $\dim((\Ext^{s+2}_R(M,R/\fr p)))\leqslant 1$ and therefore
$$\Supp(H^{s+2}_I(M,R/\fr p))=\Supp(\Ext^{s+2}_R(M,R/\fr p))$$
is finite. The proof of Theorem \ref{T34}$(a)$ is complete.\hfill\ $\square$

\medskip

To prove Theorem \ref{T34},$(b)$ we need one lemma more. 

\begin{lemma}\label{P29} Let $d=\dim(R)$. Assume that $\pd_R(M)<+\infty$. Then the following statements are true. 
\begin{itemize}
\item[$(i)$] $H^d_I(M,N)$ is Artinian.
\item[$(ii)$] $\Supp(H^{d-1}_I(M,N))$ is a finite set.
\end{itemize}
\end{lemma}
\begin{proof} $(i)$ First, we claim that $H^d_I(M,R)$ is Artinian. By using the spectral sequence as in the proof of Theorem \ref{L204}, we obtain  a finite filtration
$$0=\phi^{d+1}H^d\subseteq\phi^dH^d\subseteq\ldots\subseteq\phi^1H^d\subseteq\phi^0H^d=H^d$$
of the module $H^d=H^d_I(M,R)$ such that $E^{i,d-i}_{\infty}\cong\phi^iH^d/\phi^{i+1}H^d$ for all $0\leqslant i\leqslant d$; and for any $i=0,\ldots,d$ there exists exact sequences $$0\mt\phi^{i+1}H^t\mt\phi^{i}H^t\mt E^{i,t-i}_{\infty}\mt 0.$$
Note that $E^{i,d-i}_{\infty}$ is a subquotient of $E^{i,d-i}_2$ for all $0\leqslant i\leqslant d$, where $$E^{i,d-i}_2=H^i_I(\Ext^{d-i}_R(M,R)).$$ Thus, by using the above exact sequences, in order to prove the Artinianness of $H^d_I(M,R)$, we need only to show that $H^i_I(\Ext^{d-i}_R(M,R))$ is Artinian for all $0\leqslant i\leqslant d$. Note that $\pd_R(M)\leqslant d$ and $\Ext^j_R(M,R)=0$ for all $j>\pd_R(M)$. Therefore $\dim(\Ext^{d-i}_R(M,R))\leqslant i$ for all $0\leqslant i\leqslant d$ by  Lemma \ref{L202}. It follows by \cite[Theorem 7.1.6]{Bs} that $H^{i}_I(\Ext^{d-i}_R(M,R))$ is Artinian for all $0\leqslant i\leqslant d$ and  the claim is proved.

Now, there exists an exact sequence $0\mt L\mt R^n\mt N\mt 0$, where $n$ is an integer and $L$ is a finitely generated $R$-module. Hence, in view of Theorem \ref{L204}, we get an exact sequence $$H^d_I(M,R^n)\mt H^d_I(M,N)\mt 0.$$ Since $H^d_I(M,R^n)\cong H^d_I(M,R)^n$, so we get by the claim that $H^d_I(M,R^n)$ is Artinian. Therefore $H^d_I(M,N)$ is Artinian.
\medskip

\noindent
$(ii)$ First, we show that $\Supp(H^{d-1}_I(M,R))$ is finite. By similar arguments as in the proof of $(i)$,  we need only to show that $\Supp(E^{i,d-1-i}_2)$ is finite for all $0\leqslant i\leqslant d-1$, where $E^{i,d-1-i}_2=H^i_I(\Ext^{d-1-i}_R(M,R))$. Indeed,  consider two cases:  

\noindent
The first case: $0\leqslant\pd_R(M)<d-1$. If $0\leqslant i<d-1-\pd_R(M)$ then $\pd_R(M)<d-1-i\leqslant d-1.$ It implies $\Ext^{d-1-i}_R(M,R)=0$, so that $E^{i,d-1-i}_2=0$ for all $0\leqslant i<d-1-\pd_R(M)$. If  $d-1-\pd_R(M)\leqslant i\leqslant d-1$ then $0\leqslant d-1-i\leqslant\pd_R(M)$. Thus, by Lemma \ref{L202}, $\dim(\Ext^{d-1-i}_R(M,R))\leqslant i+1$. It implies by Lemma \ref{rm1} that $\Supp(E^{i,d-1-i}_2)$ is finite for all $d-1-\pd_R(M)\leqslant i\leqslant d-1$. Therefore $\Supp(E^{i,d-1-i}_2)$ is finite for all $0\leqslant i\leqslant d-1$. 

\noindent
The second case: $\pd_R(M)\geqslant d-1$. By Lemma \ref{L202}, $\dim(\Ext^{d-1-i}_R(M,R))\leqslant i+1$ for all $0\leqslant i\leqslant d-1$. Hence, by Lemma \ref{rm1}, $\Supp(E^{i,d-1-i}_2)$ is finite for all $0\leqslant i\leqslant d-1$ and the conclusion follows.
\medskip

Next, we prove that $\Supp(H^{d-1}_I(M,N))$ is finite. As in the proof of (i) we get an exact sequence
$$H^{d-1}_I(M,R^n)\mt H^{d-1}_I(M,N)\mt H^d_I(M,L).$$
Since $\Supp(H^{d-1}_I(M,R^n))$ is a finite set by the claim above and $H^d_I(M,L)$ is Artinian by (i), it follows from the exact sequence above that $\Supp(H^{d-1}_I(M,N))$ is finite as required.
\end{proof}

\noindent{\bf Proof of Theorem \ref {T34}}$(b)$. 

Let $d=\dim(R)$. Let $s$ be the least integer satisfying one of three equivalent conditions in Theorem \ref{T34}$(a)$. Then, by Lemma \ref{P29}, we have $s<d-1.$ By Theorem \ref{T24} to prove $\max\{r,\gamma\}\leqslant s$, where $r=\gdepth(I_M,N)$ and  $$\gamma=\sup\{\pd_{R_{\fr p}}(M_{\fr p})\mid\fr p\in\Supp(N/I_MN), \dim(R/\fr p)\geqslant 2\},$$
 we have only to show  that $\gamma\leqslant s$. Indeed, assume that  $\gamma>s$, then there exists $\fr p\in\Supp(N/I_MN)$ such that $\dim(R/\fr p)\geqslant 2$ and $\pd_{R_{\fr p}}(M_{\fr p})> s$. From this, we get by \cite[Corollary 19.5]{Ei} that  $\Tor_{s+1}^{R_{\fr p}}(M_{\fr p},k(\fr p))\not=0$. So, by \cite[Proposition 1.3.1]{Bh}, we get $\Ext^{s+1}_R(M,R/\fr p)_{\fr p}\not=0$. It follows by \cite[Theorem 144]{Kap} that $\Supp(\Ext^{s+1}_R(M,R/\fr p))$ is an infinite set. Moreover, since $I_M\subseteq\fr p$, we have by Lemma \ref{L21} that $H^{s+1}_I(M,R/\fr p)=\Ext^{s+1}_R(M,R/\fr p)$. It follows that  $\Supp(H^{s+1}_I(M,R/\fr p))$ is infinite. This contradicts with the choice of $s$. Thus $\gamma\leqslant s$ as required. \hfill$\square$ 

\begin{remark} $(i)$ In general, there does exist the integer $s$ in Theorem \ref{T34}$(a)$ if $\pd_R(M)=+\infty$. For example, let $k$ be a field and $R=k[[x,y,u,v]]/\fr a$, where $\fr a=(x^2,y^2)$. Set $\fr p=(x,y)/\fr a$. It is clear that $R_{\fr p}$ is not integral domain, and thus $R_{\fr p}$ is not a regular local ring. Hence $\pd_{R_{\fr p}}(k(\fr p))=\injd_{R_{\fr p}}(k(\fr p))=+\infty$, where $\injd_{R_{\fr p}}(k(\fr p))$ is the injective dimension of $R_{\fr p}$-module $R_{\fr p}/\fr pR_{\fr p}$. Now, let $M=N=R/\fr p$ and $I=\fr p$. Hence $\pd_R(M)=\injd_R(N)=+\infty$ and $\Ext^j_R(M,N)_{\fr p}=\Ext^j_{R_{\fr p}}(k(\fr p),k(\fr p))\not=0$ for all $j\geqslant 0$. Thus $\Supp(H^j_I(M,N))=\Supp(\Ext^j_R(M,N))$ is an infinite set for all $j\geqslant 0$.
\medskip

\noindent
$(ii)$ We can not replace the condition that $\fr p$ runs through  the set $\Supp(N)$ in statement $(iii)$ of Theorem \ref{T34}$(a)$ by the condition that $\fr p$ runs through the set $\Ass(N)$. Indeed, let $k$ be a field and $R=k[[x,y,u,v]]$. Set $I=(x,y)R$ and $M=R/(y)$, then $I_M=I$. Since $R$ is a regular local ring, $\pd_R(M)<\infty$. Let $N=R/(x)\cap (x^2,u)\cap (x^2,y,u^2).$ Then $\dim(N)=3$ and $\Ass(N)=\{xR, (x,u)R, (x,y,u)R\}$. It is clear that $y\in I_M$ is a generalized regular element of $N$, and $\ann(N/yN)\supseteq (x^2,y)$. Thus, for any $\fr q\in\Ass(N/yN)$ with $\dim(R/\fr q)\geqslant 2$, we have $\fr q\supseteq I_M$, so that $\gdepth(I_M,N/yN)=0$. This implies $\gdepth(I_M,N)=1$. Moreover, it is easy to see that $H^0_I(M,R/(x))=0$, $H^0_I(M,R/(x,u))=0$ and $\dim(H^0_I(M,R/(x,y,u)))=1.$ Hence $\Supp(H^0_I(M,R/\fr p))$ is finite for all $\fr p\in\Ass(N)$. However, we get by Theorem \ref{T24} that $\Supp(H^1_I(M,N))$ is not finite.
\end{remark}


\begin{thebibliography}{99}

\bibitem{Za} M.H. Bijan-Zadeh, {\it A common generalization of local cohomology theories,} Glasgow Math. J. {\bf 21} (1980), 173-181.

\bibitem{Bs} M.P. Brodmann, R.Y. Sharp, ``Local cohomology: an algebraic introduction with geometric applications," {\it Cambridge University Press,} (1998).

\bibitem{Bh} W. Bruns, J. Herzog, ``Cohen-Macaulay rings," {\it Cambridge University Press,} (1998).

\bibitem{CH}  N.T. Cuong, N.V. Hoang, {\it Some finite properties of generalized local cohomology modules,} East-West J. Math., (2) {\bf 7} (2005), 107-115.

\bibitem{Ei} D. Eisenbud, ``Commutative algebra with a view toward algebraic geometry," {\it Springer,} (1995).

\bibitem{Her} J. Herzog,  {\it Komplexe, Aufl\"osungen und dualit\"at in der localen Algebra, Habilitationsschrift,} Universit\"at Regensburg, 1970.

\bibitem{HZ} J. Herzog, N. Zamani, {\it Duality and vanishing of generalized local cohomology}, Arch. Math. J. (5) {\bf 81} (2003), 512-519.    

\bibitem{Hu} C. Huneke, {\it Problems on local cohomology,} in: Free Resolution in Commutative Algebraic Geometry, Bartlett, Boston, MA, {\bf 2} (1992), 93-108. 

\bibitem{HS} C. Huneke and R. Y. Sharp, {\it Bass numbers of local cohomology modules,} Trans. AMS, {\bf 339} (1993), 765-779.

\bibitem{Kap} I. Kaplansky , ``Commutative ring", {\it University of Chicago Press (revised edition)}, (1974).

\bibitem{Ka} M. Katzman, {\it An example of an infinite set of associated primes of local cohomology module,} J. Alg., {\bf 252} (2002), 161-166.

\bibitem{Kh} K. Khashyarmanesh, Sh. Salarian, {\it On the associated primes of local cohomology modules,} Comm. Alg., {\bf 27} (1999), 6191-6198.

\bibitem{Mar} Th. Marley, {\it Associated primes of local cohomology module over rings of small dimension,} Manuscripta Math. (4) {\bf 104} (2001), 519-525.

\bibitem{Nh1} L. T. Nhan, {\it On generalized regular  sequences and the finiteness for associated primes of local cohomology modules,}  Comm. Alg., {\bf 33} (2005), 793-806.

\bibitem{Rot} J. Rotman, ``Introduction to homological algebra", {\it Academic Press,} (1979).

\bibitem{Su}  N. Suzuki, {\it On the generalized local cohomology and its duality,} J. Math. Kyoto Univ., {\bf 18} (1978), 71-78.

\bibitem{Va} W. Vasconcelos, {\it Divisor theory in module categories,} North-Holand, Amsterdam, (1974).

\bibitem{Ya1} S. Yassemi, {\it Generalized section functors,} J. Pure  Appl. Alg., {\bf 95} (1994), 103-119. 

\end{thebibliography}
\end{document}